\documentclass[10pt]{article}
\usepackage{amssymb, amsmath, graphicx, amsthm, hyperref, mathrsfs, cleveref, wrapfig, cite, geometry, xcolor, siunitx, array}
\usepackage{geometry}
\usepackage{geometry}
\crefname{lemma}{Lemma}{Lemmas}
\usepackage[utf8]{inputenc}
\usepackage[english]{babel}
\newtheorem{theorem}{Theorem}
\usepackage[shortlabels]{enumitem}
\newcommand{\R}{\mathbb{R}}
\newcommand{\Z}{\mathbb{Z}}
\newtheorem{lemma}{Lemma}[section]

\newcommand{\U}{\mathcal{U}}
\newcommand{\la}{\langle}
\newcommand{\ra}{\rangle}
\newcommand{\E}{\mathcal{E}}

\renewenvironment{proof}[1][\proofname]{
\par\pushQED{\qed}\normalfont
\trivlist\item[\hskip\labelsep\bfseries#1\@]
\ignorespaces}{}
\numberwithin{equation}{section}
\theoremstyle{definition}
\newtheorem{definition}{Definition}[section]
\usepackage[toc,title,page]{appendix}

\title{Existence of Optical Vortex Solitons in Photorefractive Media}
\author{Tianyi Zhang$^1$ \ Luciano Medina$^2$  \ Zihan Zhang$^2$ }
\date{%
    $^1$Fu Foundation School of Engineering and Applied Science, Columbia University, New York, New York 10027\\%
    $^2$Courant Institute of Mathematical Sciences, New York University,
    New York, New York 10012\\[2ex]%
    \today
}

\begin{document}

\maketitle

\begin{abstract}
Optical propagation and vortices in nonlinear media have been intensively studied in modern optical physics. In this paper, we establish constraints regarding the propagation constant and provide an existence theory and numerical computations for positive exponentially decaying solutions for a class of ring-profiled solitons in a type of nonlinear media known as a photorefractive nonlinearity. Our methods include constrained minimization and finite element formalism, and we study the vortex profile and its propagation by fixing the energy flux.
\end{abstract}

\section{Introduction}
Vortices has long been intensely studied in may areas in physics and mathematics. In particular, quantum vortices, also known as Abrikosov vortex or fluxon, are vortex solutions of $2$- or $3$-dimensional wave equations first predicted by Lars Onsager in the context of superfluid. Mainly developed by Alexi Abrikosov, vortex solutions later arose from the Ginzburg-Landau equations in the context of superconductivity and served as a keystone in the explanation of type--II superconductors using the Ginzburg-Landau equations which describes the phase transition of superconductors using variation principle. 

In the studies of optical beams, Chiao, Garmire and Townes~\cite{ccw6} developed their monumental work in 1964 on the nonlinear propagation of light. They derived the nonlinear Schr\"odinger equation from simplifying the Ginzburg-Landau equations in $1+1$-dimensions and demonstrated the existence of self-reinforced wave packets, later known as ``solitons"~\cite{term}. It is the phenomenon of nonlinear propagation of high-intensity light beams enabling the light beams to produce their own waveguide and propagate without spreading. They then proposed that that, in $2$- or $3$-dimensions, the the propagation wave vortices could also shown self-reinforced or soliton behaviors. 

Intuitively, vortex solitons are phase singularities where energy flows around a point. Approaching the center of the vortex, the velocity reaches infinity, and yet due to the finite energy flux, the intensity of the center will vanish~\cite{vortex}. One may essentially picture this kind of ring-profiled vortex soliton as a ring of light around a black spot directed along its axis of propagation~\cite{retracted}. Such vortices have widespread applications in many areas of mathematics and physics, including condensed matter physics, particle interactions, cosmology, superconductivity, quantum information processing, and wireless communication~\cite{retracted, ccw2, ccw3, ccw4, vortex, ccw15, ccw24, ccw26, ccw29}, and have since been intensively studied in optics, both theoretically and experimentally~\cite{ccw7, ccw9, ccw11, ccw18, ccw21, ccw28, ccw30}, and observed in various nonlinear media~\cite{ccw2, ccw3, ccw24, ccw26, ccw29}. 

In this paper, we establish a rigorous mathematical theory for a new type of nonlinearity encountered in an experiment reported by Fleischer et al.~\cite{photonic lattice}, known as a photorefractive nonlinearity~\cite{nonlinear1, nonlinear2}. Photorefractive effect essentially describes the behavior of certain materials responding to light by altering their refractive index~\cite{refra}. Similar applied analysis papers have pursued the existence of this nonlinearity in various media in an attempt to determine the geometrical properties and constraints of the potentially observable vortex solitons, and a number of bounds have been derived for the energy flux and propagation constants~\cite{short, q hall effect medina, saturable medina, retracted, cubic}. Proving the theoretical existence of the photorefractive nonlinearity will provide insights on the future direction of experimental research.

In modern optics, light is theoretically described by a complex-valued wave function governed by the nonlinear Schr\"odinger equations~\cite{ccw1, ccw8, ccw16, ccw17, ccw19, ccw20, ccw23, ccw27}. In the experimental setup used by Fleischer et al.~to create photonic lattice solitons via nonlinear optical induction~\cite{photonic lattice}, the wave function is written as two coupled nonlinear differential equations: 
\begin{align}
\begin{gathered}
i \phi_z + \frac{1}{2 k_1} \nabla_\perp^2 \phi - \Delta n_1(I) \phi = 0, \\
i \psi_z + \frac{1}{2 k_2} \nabla_\perp^2 \psi - \Delta n_2(I) \psi = 0.
\end{gathered}
\end{align}
In these equations, $\psi$ describes the slowly varying amplitude of the lattice wave and $\phi$ is the soliton-forming probe; both are complex-valued functions. $k_1, k_2$ relate to the anisotropy of the indices of refraction. Here, $\nabla^2_\perp$ denotes the Laplace operator over the transverse plane of the coordinate. $\Delta n_1, \Delta n_2$ are the nonlinear index changes induced by the total intensity $I = |\phi|^2+|\psi|^2$. Fleischer et al. chose a photorefractive screening technique in which the change in index is given by
\begin{align}
\Delta n_1 (I) = \frac{P}{1 + I} \quad \text{and} \quad \Delta n_2 (I) = \frac{Q}{1 + I}, \label{deltan}
\end{align}
where $P, Q \ne 0$ denote coupled parameters \cite{photonic lattice2}.

As our focus here is a ring-profiled spatial soliton, we expect the field to be written in polar coordinates over $\R^2$ using an $m$-vortex ansatz~\cite{yang, ccw25} as
\begin{align}
\phi(r) = e^{i(\beta_1 z + m_1\theta)} u(r), \quad \psi(r) = e^{i(\beta_2 z+ m_2\theta)} v(r),
\end{align}
where $\beta_1, \beta_2\in \R$ are the wave propagation constants, $\theta$ represent the phase, and $m_1, m_2\in\Z$ are known as vortex numbers, or topological charges of vortices. $u, v$ are real-valued, radial profile functions which gives rise to the amplitude of the solitons. The presence of the the vortex at $r=0$ requires $u(0) = v(0)= 0$, and as in~\cite{ccw25}, the self-concentrating effect of solitons suggest $u(r), v(r)$ may be assumed to vanish at a large distance $R>0$~\cite{yang}. Thus, we here present the boundary condition $v(R) = u(R) = 0$ for some $R>0$. 

Using the changes in index given by \eqref{deltan}, we arrive at a boundary value problem for the coupled nonlinear differential equations
\begin{align}
    \begin{gathered}
    (ru_r)_r - \frac{m_1^2}{r} u - \frac{ru}{1 + (u^2 +v^2)} = \beta_1 ru,\quad 0 < r < R ,\\
    (rv_r)_r - \frac{m_2^2}{r} v - \frac{rv}{1 + (u^2 +v^2)} = \beta_2 rv,\quad 0 < r < R
    \end{gathered} \label{45}
\end{align}
satisfying the boundary conditions
\begin{align}
u(0)=u(R)=v(0)=v(R)=0.\label{1.4cond}
\end{align}
In this paper, we prove the existence of semitrivial solutions of the coupled system \eqref{45}--\eqref{1.4cond} by simplifying the coupled system into a single ordinary differential equation: these are cases where $v$ is a scalar multiple of $u$. Here, we use the substitution $v = \sqrt{\alpha - 1}\ u$, $\alpha \ge 1$, for convenience. In this case, as we will later prove in Lemma \ref{2sol}, if both $u, v$ are nonzero, $|m_1| = |m_2|$ and $\beta_1 = \beta_2$. If either $u, v$ is equal to $0$, the problem is already a single equation with $\alpha=1$. Hence, we can simplify \eqref{45} as a two-point boundary value problem for the nonlinear ordinary differential equation
\begin{align}
(ru_r)_r - \dfrac{m^2}{r} u - \dfrac{r u}{1 + \alpha u^2} &= \beta r u,
\label{equ}
\end{align}
with the boundary condition
\begin{align}
u(0)=u(R) &= 0.
\label{cond0}
\end{align}
with an undetermined parameter $\beta$ and prescribed $R$, for any given $m\in\mathbb{Z}$

Recall the physical meaning of $u$ as the soliton amplitude. We are interested in the nontrivial solutions which are classic, positive solutions of $u$. 
\begin{definition}[Positive Solutions]
    A positive solution of \eqref{equ} and \eqref{cond0} means $u \in C^2[0, R]$ is a classic solution to \eqref{equ} and \eqref{cond0} such that
    \begin{align}
    u(r) > 0\ \text{ for all }\ r\in(0, R).\label{positive}
    \end{align}
    By classic solution, we mean $u$ satisfy \eqref{equ} and \eqref{cond0} pointwise almost everywhere. 
\label{PosSol}
\end{definition}

    
To tackle this problem, we shall use the methods of calculus of variation which will be discussed in section 3. This relies on a constrained minimization approach that views \eqref{equ} as an eigenvalue problem and that the propagation constant $\beta$ arises as a Lagrange multiplier. In this sense, the differential equation can be viewed as the Euler--Lagrange equation of the action functional
\begin{align}
I(u) = \dfrac12\int_0^R \left\{ r u_r^2 + \dfrac{m^2}{r}u^2 + \frac{r}{\alpha}\ln(1 + \alpha u^2) \right\} \, dr \label{eqI}
\end{align}
subject to the constraint functional 
\begin{align}
P(u) 
= 2\pi\int_0^R r u^2 \, dr = P_0,
\label{eqP}
\end{align}
where $P_0>0$ is fixed and $\beta$ appears via the Lagrange multiplier of the constrained optimization problem
\begin{align}
I_0 = \inf_{u\in\U} \left\{ I(u)| P(u) = P_0 \right\},\label{minProb}
\end{align} 
with $\U$ denoting an admissible class
\begin{align}
\U = \{u(r) \text{ absolutely continuous over } [0, R], u(0)=u(R)=0, \E(u)<\infty \} \label{Uclass}
\end{align}
and the energy functional defined by
\begin{align}
\E (u) = \int_0^R \left\{ ru_r^2 + \dfrac{m^2}{r} u^2 \right\}\, dr.
\label{eqE}
\end{align}
The functional $P(u)$ has been used as the beam power~\cite{ccw17, beam}, energy flux~\cite{ccw27}, or stability integral~\cite{stab integral}. In our case, we refer to it as the energy flux. {
Similar to the method described in \cite{yang}, it suffices to show that and the optimal solution $u\in \mathcal{U}$ under action functional \eqref{eqI} is a positive solution satisfying \eqref{equ} and \eqref{cond0} pointwise. 
}

Our major results are as follows:
\begin{theorem}
For the existence of positive solutions of \eqref{equ}--\eqref{cond0}, the following conditions are necessary:
\begin{enumerate}[1.]
\item 
\begin{align}
\beta &< - \dfrac{m^2 + r_0^2}{R^2},
\label{upper}
\end{align}
where $r_0 \approx 2.404825$ is the first zero of the Bessel function $J_0$~\cite{yang}. 
\item
\begin{align}
\max_{0<r<R}\{u(r)\}^2 > -\frac1{\alpha}\left[
\left(\beta + \dfrac{m^2}{R^2}\right)^{-1}+1\right]
\label{maxu}
\end{align}
with $\max_{0<r<R}\{u(r)\}$ denoting the peak of the profile curve.
\item
If 
\begin{align}
\beta > -\dfrac{m^2}{R^2} - 1, \label{1.15}
\end{align}
then the solution $u$ is bounded by the exponentially decaying estimate
\begin{align}
u(r)^2 \le C_{\epsilon_{0}} e^{-\sqrt{\epsilon}_0 r}
\label{expon}
\end{align}
as $r\to R^-$. Here, $\epsilon_0$ is a constant bounded by $\epsilon_0 > 2\left(\beta + \dfrac{m^2}{R^2} + 1 \right)> 0$, and $C_{\epsilon_0}$ depends on $\epsilon_0$.
\end{enumerate}
\end{theorem}

\begin{theorem}
There exists solution pair $(u, \beta)$ satisfying the differential equations \eqref{equ}--\eqref{cond0} pointwise, where $u$ is a positive solution following Definition \ref{PosSol} and $\beta \in \R$. Such a solution can be obtained by solving the constrained minimization problem \eqref{minProb}, from which $\beta$ arises as a Lagrange multiplier.
\end{theorem}

The remainder of this paper is structured as follows. In section 2, we provide proofs for Theorems 1. In section 3, we prove our Theorem 2, which is the main theorem demonstrating the existence of positive solutions using the constrained minimization approach. Section 4 presents a selection of results from numerical computations using the finite element method and provides a summary of the results. The proof of a lemma used in the paper is provided in the appendix.

\section{Proof of Theorem 1}

To prove \eqref{upper}, multiplying by $u$ and integrating by parts over $r$, we obtain
\begin{align}
    \left.ru u_r\right|^R_0 - \int_0^R ru_r^2 \, dr &= \int_0^R \left( \beta + \dfrac{m^2}{r^2} + \dfrac{1}{1 + \alpha u^2 }\right) ru^2\, dr. \label{uint}
\end{align}

We first want to show that the first term is $0$. Recall that we are interested in a positive solution as defined in Definition \ref{PosSol}. Since $u\in C^2[0, R]$, as $r\to R$, $u_r$ is uniformly bounded which implies $ru u_r \to 0$ as $r\to R$. We will then prove that $ru(r)u_r(r)\to 0$ as $r\to 0$. Suppose that $\liminf_{r\to0}\{ ru|u_r| \}\ne 0$. Then, there exists some $\epsilon>0$ and $ \delta \in (0, R]$ such that $ru|u_r| > \epsilon$ for any $r \in (0, \delta)$. Using the Cauchy--Schwarz inequality, we have that
\begin{align}
    \left(\int_0^\delta ru_r^2\, dr\right)^{1/2} \left(\int_0^\delta \dfrac{u^2}{r}\, dr\right)^{1/2} \ge \int_0^\delta u|u_r|\, dr > \int_0^\delta \dfrac{\epsilon}{r}\, dr = \infty.
\end{align}
This is a contradiction with regard to the energy functional, which implies that \\$\liminf_{r\to0}\{ ru|u_r|\}= 0$. Hence, we can establish the following inequality:
\begin{align}
    - \int_0^R ru_r^2 \, dr \ge \int_0^R \left( \beta + \dfrac{m^2}{r^2} 
    + \dfrac{1}{1 + \alpha u^2 }\right) ru^2 \ dr > \left( \beta + \dfrac{m^2}{R^2} 
    \right) \int_0^R ru^2 \ dr. 
    \label{cond1}
\end{align}
Recall Poincaré's inequality\cite{evans} 
\begin{align}
    \dfrac{R^2}{r_0^2} \int_0^R r u_r^2 \, dr \ge \int_0^R ru^2 \, dr
    \label{16}
\end{align}
where $r_0 \approx 2.40483$ is the first zero of the Bessel function $J_0$. Together with \eqref{cond1}, we have that
\begin{align}
    - \dfrac{{r_0}^2}{R^2}\int_0^R ru^2 \, dr &> 
    \left( \beta + \dfrac{m^2}{R^2} \right) \int_0^R ru^2 \ dr, 
\end{align}
which results in $\beta < - \dfrac{m^2 + r_0^2}{R^2}$, as claimed.
\qed


To prove \eqref{maxu}, let $u \in C^2[0, R]$ satisfy boundary condition \eqref{cond0}. As we are interested in nontrivial solutions of $u$, we know that there exists a maximizer $\eta\in (0, R)$ such that $ u(\eta) = \max\{u\}$, which also implies $u_{rr}(\eta) \le 0$ and $ u_r(\eta) = 0$. At $r=\eta$, \eqref{equ} becomes
\begin{align}
    \eta u_{rr}(\eta) - \dfrac{m^2}{\eta^2} \eta u(\eta) 
    - \dfrac{\eta u(\eta)}{1+ \alpha u(\eta)^2} &= \beta \eta u(\eta),
\end{align}
which simplifies to
\begin{align}
    u_{rr}(\eta) &= \left( \beta + \dfrac{m^2}{\eta^2} +
    \dfrac{ 1}{1+ \alpha u(\eta)^2} \right) u(\eta).
\end{align}
Taking the concavity at the peak $u_{rr}(\eta)\le 0$, we arrive at the inequality
\begin{align}
    0 \ge \left( \beta + \dfrac{m^2}{\eta^2} + \dfrac{ 1}{1+ \alpha u(\eta)^2} \right) u(\eta).
    \end{align}
This simplifies to the following expression for the maximum of $u$, as claimed:
\begin{align}
    u(\eta)^2 \ge -\frac{1}{\alpha}\left(\dfrac{1}{\beta+\dfrac{m^2}{\eta^2}}+1\right)>
    -\frac{1}{\alpha}\left(\dfrac{1}{\beta+\dfrac{m^2}{R^2}}+1\right).
\end{align}
Furthermore, we see that if an desired inequality
\begin{align}
    \beta > - \dfrac{m^2}{R^2} - 1,
\end{align}
holds between the prescribed variables $\beta, m, R$, then $\max\{u\}>0$. This can serve as a sufficient but not necessary condition for the existence of nontrivial solution. 
\qed 


To prove \eqref{1.15}--\eqref{expon}, let $\Delta$ denote the Laplacian operator in polar coordinates, defined by 
\begin{align}
\Delta = \dfrac{1}{r}\dfrac{\partial}{\partial r} \left(r\, \dfrac{\partial}{\partial r}\right).
\end{align}
Rewriting \eqref{equ} by replacing $(ru_r)_r$ with $r \Delta u$ and simplifying gives
\begin{align}
\Delta u(r) &> \left(\beta + \dfrac{m^2}{R^2} + \dfrac{1}{1+ \alpha u(r)^2}\right) u.
\end{align}
Using the product rule $\Delta u^2 \ge 2u \Delta u $, we then obtain
\begin{align}
\Delta u(r)^2 &\ge 2\left(\beta + \dfrac{m^2}{R^2} + \dfrac{1}{1 + \alpha u(r)^2} \right) u(r)^2. \label{34}
\end{align}
Because $\lim_{r\to R^-} \dfrac{1}{1+ \alpha u^2} = 1$, for all $\epsilon > 0 $, there exists $R'\in (0, R)$ such that 
\begin{align}
\dfrac{1}{1+ \alpha u(r)^2} > 1 - \epsilon \label{35}
\end{align}
for all $r \in (R', R)$. Substituting \eqref{35} into \eqref{34}, we then have
\begin{align}
\Delta u(r) ^2 &> 2\left(\beta + \dfrac{m^2}{R^2} + 1 - \epsilon \right) u(r)^2.
\end{align}
Insert the condition $\beta > -\dfrac{m^2}{R^2} - 1$ and select a small $\epsilon$ such that $\epsilon < \beta + \dfrac{m^2}{R^2} + 1 $. It can then be shown that there exists some $R'$ such that, for all $r\in [R', R]$,
\begin{align}
\Delta u(r)^2 &> \epsilon_0 u(r)^2 \label{37}
\end{align}
for some $\epsilon_0>0$. Consider the decreasing exponential function $\xi$ with respect to its exponential coefficient $\sqrt{\epsilon_0}$ and a constant $C>0$ defined as
\begin{align}
\xi(r) = Ce^{-\sqrt{\epsilon_0} r}, 
\label{38}
\end{align}
for which the Laplacian is computed to be
\begin{align}
\Delta \xi(r) = \left(\epsilon_0 - \frac{\sqrt{\epsilon_0}}{r}\right) \xi(r).
\end{align}
Subtracting \eqref{37} from \eqref{38} gives 
\begin{align}
\Delta(u(r)^2 - \xi(r)) &> \epsilon_0 \left(u(r)^2 - \xi(r) \right) + \frac{\sqrt{\epsilon_0}}{r}\, \xi(r).
\end{align}
Because $u^2(r) \ge u^2(R)=0$ and $C e^{-\sqrt{\epsilon_0} r} > C e^{-\sqrt{\epsilon_0} R}$ for all $r\in [R', R]$, we find that 
\begin{align}
u^2(r) - \xi(r) < u^2(R) - \xi(R) = -\xi(R) < 0.
\end{align}
This establishes the following inequality with respect to exponential decay:
\begin{align}
u^2(r) < Ce^{-\sqrt{\epsilon_0} r}, \quad r \in [R', R].
\end{align}
We further denote this as $u^2(r) < C_{\epsilon_0}e^{-\sqrt{\epsilon_0} r},\ r\in[R_{\epsilon_0}, R]$ to emphasize the dependence of $C, R$ on $\epsilon_0$, similar to \cite{saturable medina}.
\qed

\section{Proof of Theorem 2}
To prove Theorem 2, we employ a variational principle and constrained minimization problem. Recall the definitions of the action functional \eqref{eqI}, constraint functional \eqref{eqP}, admissible class \eqref{Uclass}, and the minimization problem \eqref{minProb}. 

The differential equation \eqref{equ} can be viewed as an eigenvalue problem where the propagation constant $\beta$ is undetermined and acts as a Lagrange multiplier. In order to prove the existence of a solution pair $(u, \beta)$, it suffices to show that a solution to the minimization problem \eqref{minProb} exists subjected to the prescribed value of energy flux $P_0$. 

We first shows the existence of feasible solution to the constraint minimization problem \eqref{minProb} by showing the existence of $u\in\mathcal{U}$ such that $P(u) = P_0$ and $I(u) < \infty$. Since $x \ge \dfrac{1}{\alpha} \ln (1 + \alpha x)$ for $x\ge0, \alpha>0$, we see that 
\begin{align}
    I(u) \le \frac12 \int_0^R \left\{ ru_r^2 + \frac{m^2}{r}u^2 + ru^2\right\} dr
    = \frac12 \int_0^R \left\{ru_r^2 + \frac{m^2}{r}u^2\right\} dr + \frac{P_0}{4\pi}.
    \label{ineqIP}
\end{align}
Consider the following function $f\in\mathcal{U}$ defined as
\begin{align}
    f(r) = \sqrt{\frac{30 P_0}{\pi R^6}}\ r(R-r).
\end{align}
Since the function was designed so that 
\begin{align*}
    P(f) = 2\pi \int_0^R r f(r)^2\, dr = \frac{30 P_0}{\pi R^6}\frac{\pi b^2R^6}{30} = P_0, 
\end{align*}
along with the property $f(0) = f(R) = 0$, we see that $f$ belongs to the feasible set of \eqref{minProb}. By direct computations, we can compute for $I(f)$. Using the following computations 
\begin{align*}
    \int_0^R rf_r^2 \, dr = \frac{b^2 R^4}{4}, \quad \quad
    \int_0^R \frac{f^2}{r}\, dr = \frac{b^2 R^4}{12}, 
\end{align*}
we arrive at a finite value for $I(f)$
\begin{align}
    I(f) \le \left(\frac{15 + 5m^2}{R^2} + 1\right)\frac{P_0}{4\pi}.
\end{align}
Since $f$ is a feasible solution to \eqref{minProb} with finite objective value, the optimization problem is well defined. Note the lower bound of the action functional $I\ge 0$; the boundedness of $I$ gives us a weakly convergent subsequence $\{u_k\}_{k=1}^\infty$, which we can use it as minimizing sequence of \eqref{minProb} such that
\begin{align}
I(u_k) \to I_0\ \text{ as }\ k\to \infty \ \text{ and }\ I(u_1) \ge I(u_2) \ge \hdots \ge I_0.
\end{align}

Using boundedness and the fact that $\{u_k\}_{k=1}^\infty$ is minimizing, it follows that there exist $C>0$ independent of $k$ such that 
\begin{align}
    C = I(u_1) \ge \int_0^R ru_{k,r}^2 \, dr + m^2 \int_0^R \frac{u_k^2}{r}\, dr.
    \label{Cbound}
\end{align}
This inequality is further developed into 
\begin{align}
    \begin{split}
        (1 + R^2) C &\ge (1+R^2) \left(
        \int_0^R ru_{k,r}^2 \, dr + m^2 \int_0^R \frac{u_k^2}{r}\, dr.
        \right)\\
        &\ge \int_0^R r u_{k, r}^2\, dr + R^2 \int_0^R \frac{u_k ^2}{r} \, dr 
        \\
        &\ge\int_0^R ru_{k, r}^2 + ru_{k}^2\, dr.
        \end{split}
        \label{cbound}
\end{align}

Because the distributional derivative of $u$ satisfies $||u|_r| \le |u_r|$~\cite{ellipse} and the functionals $I$ and $P$ are both even, we have 
\begin{align}
    I(u_k) \ge I(|u_k|) \text{ and } P(u_k) = P(|u_k|),
\end{align}
which implies that we may modified minimizing sequence $\{u_k\}_{k=1}^\infty$ so that it consists of nonnegative valued functions $u_k>0$.
    
By setting $r = \sqrt{x^2 + y^2}$, we can view these functions as radially symmetric functions over the disk $B_R = \{(x, y) \in \R^2 | x^2 + y^2 \le R^2\}$  which all vanish at boundary $\partial B_R$. Together with \eqref{cbound}, $\{u_k\}_{k=1}^\infty$ is shown to be bounded under the radially symmetrically reduced norm, with definition
\begin{align}
||u_k|| := \int_0^R ru_k^2 + ru_{k, r}^2 \, dr,
\end{align}
for the standard Sobolev space $W_{0}^{1, 2}(B_R)$. Notice that we are using $\{u_k\}_{k=1}^\infty$ to denote both the minimizing subsequence as well as the radial function view. Since $u_k$ are radial symmetric function, we can write it as $u=u(r)$ which satisfy $u(R)=0$. Given that the sequence is monotonically decreasing but bounded in terms of the action functional, it is sufficient to show that $u_k$ converges weakly to a minimizer $u$ on $W_0^{1, 2}(B_R)$. Here $u$ is used to denote the limit of sequence.

Using the compact embedding $W^{1, 2}(B_R) \subset \subset L^p(B_R)$ for $p\ge 1$, the sequence $u_k$ converges strongly to $u$ in $L^p(0, R)$. Furthermore, for any $\epsilon \in (0, R)$, $\{u_k\}_{k=1}^\infty$ is a bounded sequence in the space $W^{1, 2}(\epsilon, R)$. If we apply compact embedding $W^{1, 2}(\epsilon, R) \subset \subset C[\epsilon, R]$, we see that $u_k\to u$ as $k\to \infty$ uniformly over $[\epsilon, R]$. 

To show that $u(0) = 0$, let $r_1, r_2 \in (0, R)$ such that $r_1<r_2$, we have Cauchy--Schwarz inequality 
\begin{align}
    \begin{split}
        \left|u_k^2(r_2) - u_k^2(r_1)\right| &= \left|\int_{r_1}^{r_2}\left(u_k^2 (r)\right)_r\, dr \right|
        \le \int_{r_1}^{r_2} 2\, |u_{k}(r)| \, |u_{k, r}(r)| \, dr \\
        &\le 2 \left(\int_{r_1}^{r_2} r u_{k, r}^2(r)\right)^{1/2} \left(\int_{r_1}^{r_2}\frac{u_k^2(r)}{r}\right)^{1/2} 
        \le 2 C^{1/2} \left(\int_{r_1}^{r_2}\frac{u_k^2(r)}{r}\right)^{1/2}
    \end{split}
    \label{r12}
\end{align}
where the constant $C$ is provided in \eqref{Cbound}. By letting $k\to\infty$, this becomes
\begin{align}
    \left|u^2(r_2) - u^2(r_1)\right| \le 2 C^{1/2} \left(\int_{r_1}^{r_2}\frac{u^2(r)}{r}\right)^{1/2}
\end{align}
In view of \eqref{cbound} and applying Fatou's Lemma, we have
\begin{align}
\int_0^R ru_{r}^2 \, dr & \le \liminf_{k\to\infty} \int_0^R ru_{k, r}^2 \, dr \\
\int_0^R \frac{u^2}{r}\, dr & \le \liminf_{k\to\infty} \int_0^R \frac{u_k^2}{r}\, dr.
\end{align}
We see that $\dfrac{u^2}{r} \in L^p(0, R)$. This implies that the right hand side of \eqref{r12} tends to zero as $r_1, r_2 \to 0$, which further implies that the limit exists for $u(r)^2, r\to0$ and it is  
\begin{align}
    \lim_{r\to 0}u^2 (r) = 0.
\end{align}
Thus, the boundary condition for $u(0) = 0$ is achieved.
Recall \eqref{ineqIP}, the $\dfrac{r}{\alpha}\ln(1 + \alpha u^2)$ term is bounded by
\begin{align}
    0 \le \liminf_{k\to\infty} \int_0^R \frac{r}{\alpha} \ln(1+\alpha u_k^2)\, dr \le \lim_{k\to \infty} \frac{P(u_k)}{4\pi}  =  \frac{P_0}{4\pi}. 
\end{align}

Using the above results, we can conclude that $u$ obtained from minimizing sequence $\{u_k\}_{k=1}^\infty$ for the minimization problem \eqref{minProb} satisfies the desired solution, i.e.\,positive solution and satisfying the boundary conditions. The functional $I$ is weakly lower semicontinuous
\begin{align}
I_0 = I(u) \le \liminf_{k\to \infty} I(u_k),
\end{align}
and the constrains can be achieved as
\begin{align*}
    P(u) = \lim_{k\to \infty} P(u_k) = P_0.
\end{align*}
Consequently, $\beta$ appears as an Lagrange multiplier. 

Since our solution $u$ is nonnegative, suppose there exists $r_0\in(r, R)$ such that $u(r_0) = 0$, $r_0$ would be a minimizer with $u_r(r_0) = 0$. Recall the uniqueness theorem for the initial value problem of ordinary differential equations, we must have the trivial solution $u(r) = 0$ for all $r\in(0, R)$ which contradict with the value for energy flux $P_0$. Therefore the positive condition $u(r)>0, r\in(0, R)$ is proved.  

\qed

\section{Computations Using Finite Element Method}
To approximate solutions for \eqref{minProb} and their corresponding eigenvalue $\beta$ given a fixed energy flux $P_0$, we employ variational principle and a type of finite element method known as Ritz-Galerkin method to approximate solutions to boundary valued problems in weak formulation. 

Viewing the  the admissible class $\U$ defined in \eqref{Uclass} as a vector space, we can describe its elements $u$ as a linear combination of basis functions $\{\psi_j\in C^1[0, R]\}_{j=1}^\infty$ with coefficients $c^i\in\R $:
\begin{align}
u &= \sum_{j=1}^\infty c^j \psi_j.
\label{inflincom}
\end{align} 
The vector space is then equipped with the weighted inner product
\begin{align}
\la u, \Tilde{u}\ra &= 2\pi \int_0^R ru\Tilde{u}\, dr
\label{inner prod}
\end{align}
to be viewed as a Hilbert Space. To approximate the solution numerically, we will choose a set of Schauder basis\cite{schauder} so that we can approximate the solution \eqref{inflincom} using finite dimension $N\in\mathbb{N}$ as
\begin{align}
u &= \sum_{j=1}^N c^j \psi_j. \label{mincomb}
\end{align}
Since the trigonometric functions can serve as a set of Schauder basis, the specific choice of basis functions we will use for the weighted inner product is
\begin{align}
    \phi_j(r) = \sin\left(\frac{j\pi r}{R}\right). \label{basis}
\end{align}
The orthogonal basis $\{\psi_j\}_{j=1}^N$ can be obtained by applying Gram--Schmidt process on \eqref{basis} such that $\langle\psi_i, \psi_j\rangle = \delta_{ij}$. To minimize error, we have implemented the modified Gram--Schmidt\cite{mgs} in code. 
Note that the energy flux can be computed with the inner product as follows:
\begin{align}
P(u) = \sum_{i, j}^N c^i c^j \langle \psi_i, \psi_j\rangle = 
\sum_{i, j}^N c^i c^j \delta_{ij} = \sum_{j}^N {(c^j)}^2.
\end{align}
Our code obtains an orthonormal basis by applying the modified Gram--Schmidt process on $\left\{\sin\left(j \pi r / R\right)\right\}_{j= 1}^N$. As $u$ can be viewed as a minimum of the action functional \eqref{eqI}, in view of \eqref{minProb} and \eqref{mincomb}, we arrive at the optimization problem 
\begin{align}
\min \left\{
F(c) = \left.I\left(\sum_{j = 1}^N c^j \psi_j \right) \right|\ P(u)= P_0 = \sum_{j}^N {(c^j)}^2
, \ u(r) > 0 \text{ for }r\in(0, R) \right\} 
\label{opti prob}
\end{align}
where we minimize $c \in \R^N$. We use MATLAB's Optimization Toolbox and the Chebfun package to numerically solve for \eqref{opti prob}. For convenience, all subsequent numerical computations have $\alpha = 1$ as the curves with $\alpha > 1$ generally have the same shape. The propagation constant then can be computed as follows: let $\lambda\in\R$ be the Lagrange multiplier, so that 
$\la I'(u), \Tilde{u}\ra = \lambda \la P'(u), \Tilde{u} \ra$.
$I'(u)$ and $P'(u)$ are given by
\begin{align}
\la I'(u), \tilde{u} \ra &= \int_0^R \left\{ r u_r \Tilde{u}_r + \frac{m^2}{r} u\Tilde{u} + 
\frac{ru\Tilde{u}}{ 1+\alpha u^2} \right\}dr 
\\
\la P'(u), \tilde{u} \ra &= 4\pi \int_0^R ru\Tilde{u} dr,
\end{align}
where $\tilde{u}$ is an arbitrary test function. The expression $\la I'(u), \Tilde{u}\ra = \lambda \la P'(u), \Tilde{u} \ra$ can then be expanded as
\begin{align}
\int_0^R \left\{
r u_r \Tilde{u}_r + \frac{m^2}{r} u\Tilde{u} + \frac{ ru\Tilde{u}}{1+ \alpha u^2}
\right\}dr = 4\pi \lambda \int_0^R ru \Tilde{u} dr. 
\label{lag eq1}
\end{align}
Applying integration by parts to the first term, we have
\begin{align}
-\int_0^R \left\{ (r u_r)_r - \frac{m^2}{r} u - \frac{ru}{1+ \alpha u^2} 
\right\} \tilde{u} dr = 4\pi \lambda \int_0^R ru \Tilde{u} dr.
\end{align} 
This gives the weak version of \eqref{equ}. We can therefore claim that $\beta = -4 \pi \lambda$. We rewrite \eqref{lag eq1} as 
\begin{align}
\int_0^R \left\{ r u_r^2 + \frac{m^2}{r} u^2 + \frac{r u^2}{1+ \alpha u^2} 
\right\}dr = -\beta\left(\frac{ P_0 }{2\pi}\right),
\end{align}
and finally arrive at an expression for $\beta$ that is suitable for numerical computation:
\begin{align}
\beta &= -\frac{2\pi}{ P_0 } \int_0^R \left\{
r u_r^2 + \frac{m^2}{r} u^2 + \frac{r u^2}{1+ \alpha u^2} 
\right\}dr.
\end{align}
To compute the error in the numerical computations, taking \eqref{equ} and rearranging into the form
\begin{align}
0 &= (ru_r)_r - \dfrac{m^2}{r} u - \dfrac{r u}{1 + \alpha u^2} - \beta r u, 
\label{equ0}
\end{align}
as the left-hand side is $0$, we can use this to calculate the computational error of our eigenvalue by integrating over the square of the right-hand side of \eqref{equ0}.
\begin{align}
\Delta \beta &= \int_0^R \left((ru_r)_r - \dfrac{m^2}{r} u - \dfrac{r u}{1 + 
\alpha u^2} - \beta r u \right)^2 \, dr. 
\label{error}
\end{align}
As an example of the numerical results, Fig.~\ref{varm} shows the profile of the amplitude of the soliton if we fix $R=40$ and $P_0=200$ while varying the vortex number $m$. We see that increasing $m$ corresponds to an outward shift of the peak. Figure \ref{varp} demonstrates the relationship between the profile function and energy flux $P_0$. An increase in $P_0$ corresponds to an increase in the height of the peak of the profile curve. Furthermore, beyond some threshold, the profile curves start to become visibly asymmetric over $(0, R)$, which is likely to be the threshold governing the behavior as $R\to \infty$. The reason for this conjecture is that the asymmetry will cause $u$ to vanish near $R$; and thus, for an increase $\Delta R$ in $R$, the increase $\Delta I, \Delta P$ in the corresponding integrals is likely to be small due to the asymmetry over the domain. Physically, $R$ is chosen to be large enough such that $u$ is nearly vanished some distance before $R$ for a self-reinforced soliton~\cite{yang}. In this view, only the asymmetric results have applications in physics. Recall Theorem 1; the additional condition $\beta > -\frac{m^2}{R^2} - 1$, which implies exponential decay near $R$, has the potential of categorizing solutions for the $R\to \infty$ case. Unfortunately an increase in $R$ drastically increase the computation's error, dimensions required, and time to be verified with confidence. The $R\to \infty$ case could be studies more in future research. 

Some values for propagation constants $\beta$ are also computed. For example, for $R=20$, $m=1$, and $P_0 = 1, 50, 100, 200, 400, 800$, the numerically computed eigenvalues are $\beta = -1.1419$, $-0.9328$, $-0.7818$, $ -0.6185$, $-0.4792$, $-0.3679$ with corresponding errors $\Delta \beta = 2.9550$\SI{1e-4}, $5.7607$\SI{1e-4}, $4.4359$\SI{1e-4}, $3.6516$\SI{1e-4}, $5.6023$\SI{1e-4}, $0.0015$.
The following table presents some of the computed eigenvalues along with their errors for $R=20$, $m=1$, and various $P_0$.

\begin{center}
\begin{tabular}{c|c|c}
    $P_0$ & $\beta$    & $\Delta \beta$ \\
    \hline
    $1$   & $-1.1419$  & $2.9550\times 10^{-5}$\\ 
    $50$  & $-0.9328$  & $5.7607\times 10^{-4}$\\
    $100$ & $-0.7818$  & $4.4359\times 10^{-4}$\\ 
    $200$ & $-0.6185$  & $3.6516\times 10^{-4}$\\ 
    $400$ & $-0.4792$  & $5.6023\times 10^{-4}$\\ 
    $800$ & $-0.3679$  & $0.0015$
\end{tabular}
\end{center}

Finally, Fig.~\ref{lin} shows $\beta$ as a function of $P_0$ for $R=20$ and $m$ = $1$, $2$, $3$, $4$, $5$. The error bars show the computational error computed with \eqref{error} due to the finite number of dimensions ($N=20$). The corresponding upper bounds, proved with Lemma \ref{upper}, are shown as dotted lines.

\begin{figure}
\centering
\includegraphics[width=0.8\textwidth]{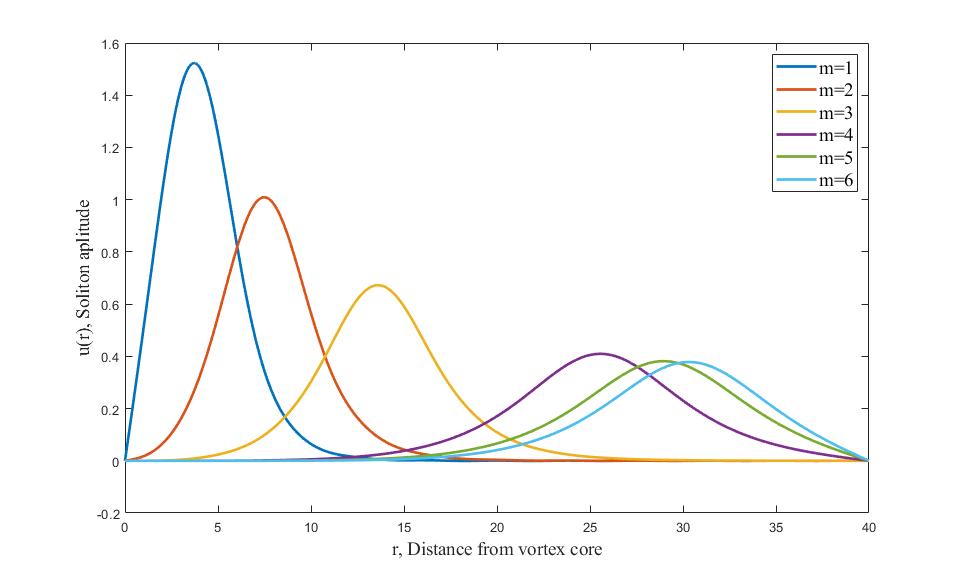}
\caption{\small{Solitons amplitude for $R=40$, $P_0=200$, and $m = 1$, $2$, $3$, $4$, $5$, $6$.}}
\label{varm}
\end{figure}
\begin{figure}
\centering
\includegraphics[width=0.8\textwidth]{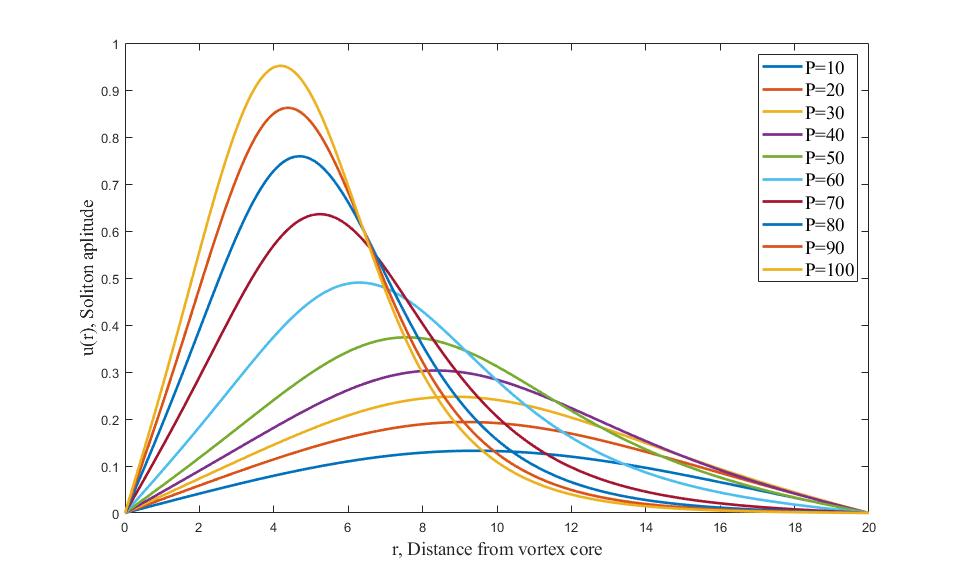}
\caption{\small{Solitons amplitude for $R=20$, $m=1$, and $P_0 = 10$, $20$, $30$, $40$, $50$, $60$, $70$, $80$, $90$, $100$.}}
\label{varp}
\end{figure}
\begin{figure}
\centering
\includegraphics[width=0.8\textwidth]{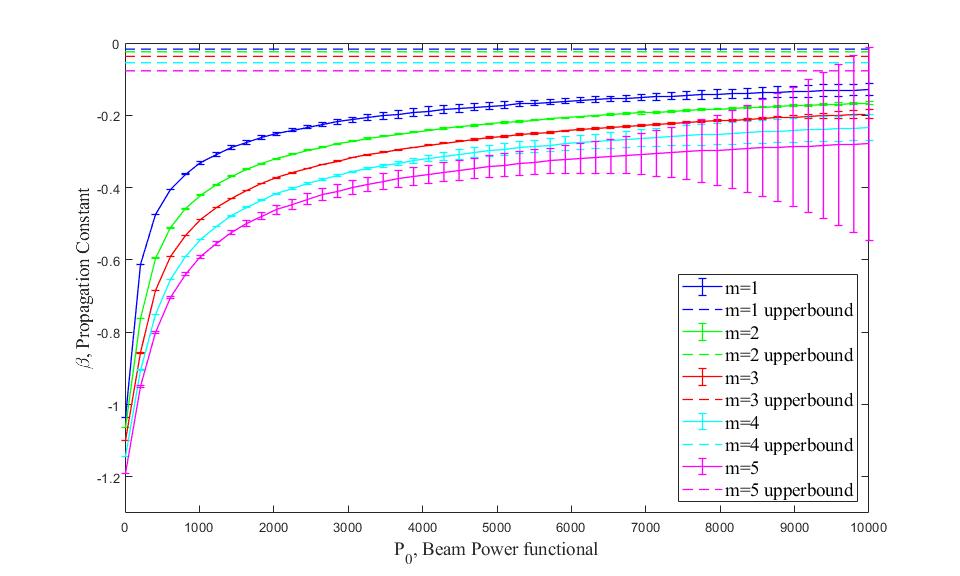}
\caption{{Propagation constant with respect to energy flux for $R=20$, $m=1$, $2$, $3$, $4$, $5$ plotted with $x$-axis. Error bars are plotted accordingly. Dashed lined are upper bounds of $\beta$ proved by Lemma \ref{upper}}}
\label{lin}
\end{figure}

\newpage 

\appendix
\renewcommand{\thesection}{\Alph{section}.\arabic{section}}
\begin{appendices}
\setcounter{section}{0}
\section{Additional Lemmas}
\begin{lemma}
If $u$ and $v$ are nonzero scalar multiples of each other that satisfy positive solutions to \eqref{45}, then $|m_1|=|m_2|$ and $\beta_1 = \beta_2$.
\label{2sol}

\end{lemma}
\begin{proof}
Suppose that both $u$ and $v$ are nonzero and positive. By making one a scalar multiple of the other, we can set $v = u\sqrt{\alpha - 1} $, $\alpha > 1$, and then \eqref{45} reduces to
\begin{align}
(ru_r)_r - \frac{m_1^2}{r} u - \frac{ru}{1 + \alpha u^2} &= \beta_1 ru, \label{u}\\
(ru_r)_r - \frac{m_2^2}{r} u - \frac{ru}{1 + \alpha u^2} &= \beta_2 ru. \label{v}
\end{align}
Subtracting \eqref{u} from \eqref{v}, we can simplify the difference as
\begin{align}
-(m_1^2 - m_2^2) &= (\beta_1 - \beta_2) r^2 u.
\end{align}
Given that $m_1, m_2, \beta_1, \beta_2$ are constants while there exist more than one different $r^2 u > 0$ in the domain, the only way to make this equation hold for all $r\in (0, R)$ is to have $0$ on both sides, implying that $|m_1| = |m_2|$ and $\beta_1=\beta_2$, as claimed.
\end{proof}


\end{appendices}


\end{document}